\documentclass[5p]{elsarticle}

\makeatletter
\def\ps@pprintTitle{%
 \let\@oddhead\@empty
 \let\@evenhead\@empty
 \def\@oddfoot{}%
 \let\@evenfoot\@oddfoot}
\makeatother

\usepackage{answers}
\usepackage{setspace}
\usepackage{graphicx}
\usepackage{multicol}
\usepackage{mathrsfs}
\usepackage{amsmath,amssymb}
\usepackage{bbold}
\usepackage{tikz}
\usetikzlibrary{backgrounds}
\usepackage[framemethod=tikz]{mdframed}
\usepackage{fancybox}

\usepackage[mathscr]{euscript}

\usepackage{schemabloc}

\usetikzlibrary{circuits}

\usepackage[english]{babel}
\usepackage[T1]{fontenc}

\usepackage{lipsum}

\usetikzlibrary{arrows, positioning, calc}

\usepackage{blox}

\usepackage{graphicx} 
\usepackage{epsfig} 
\usepackage{pgfplots}
\usepackage{pgfplotstable}
\usepackage{resizegather}
\usepackage{mdframed}

\mdfsetup{linewidth=0.75pt,
innertopmargin=.25pt,
  skipbelow=.25em,
  skipabove=.25em,
  nobreak=true
}

\usepackage{mathtools}

\usepackage{amsmath} 
\usepackage{amssymb}  
\usepackage{xcolor}
\usepackage{mathtools}
\usepackage{soul}
\usepackage{relsize}

\usepackage{amsthm}
\usepackage[ruled]{algorithm2e}
\LinesNumbered

\makeatletter
\newcommand{\removelatexerror}{\let\@latex@error\@gobble}
\makeatother
\allowdisplaybreaks[1]
\usepackage{balance}

\usepackage[normalem]{ulem}
\usepackage{xfrac}

\usepackage[running]{lineno}

\newtheorem{theorem}    {Theorem}

\newtheorem{definition} {Definition}
\newtheorem{assm} {Assumption}

\newtheorem{remark}     {\noindent Remark}

\newcommand{\nc}{\newcommand}
\nc{\R}{{\mathbb R}}
\nc{\C}{{\mathbb C}}
\nc{\Z}{{\mathbb Z}}
\nc{\N}{{\mathbb N}}
\nc{\s}{\bar{\bf s}}
\nc{\I}{{\cal I}^*}

\usepackage[colorlinks]{hyperref}

\begin{document}

\begin{frontmatter}

\title{
\Large 
A
Convolution Bound Implies Tolerance to\\ Time-variations and Unmodelled Dynamics
\tnoteref{tt1}}
\tnotetext[tt1]{
{Funding for this research was provided by the
Natural Sciences and Engineering Research Council of Canada (NSERC).}
}

\author{Mohamad T. Shahab}
\ead{
m4shahab@uwaterloo.ca
}
\author{Daniel E. Miller\corref{ath1}}
\ead{
miller@uwaterloo.ca
}
\cortext[ath1]{
Corresponding author.
}

\address{Dept. of Electrical and Computer Engineering,\\
        University of Waterloo\\
        Waterloo, ON, Canada N2L 3G1\\
        }



\begin{abstract}
Recently it has been shown, in several settings, how to carry out
adaptive control for an LTI plant so that a convolution bound 
holds on
the closed-loop behavior;
this, in turn, has been leveraged to prove robustness of the closed-loop system
to time-varying parameters and unmodelled dynamics.
The goal of this paper is to show that 
the same is true for a large class of
finite-dimensional, nonlinear
plant and controller combinations.
\end{abstract}




\end{frontmatter}







\section{Introduction}

In control system design, a common requirement is that the closed-loop
system not only be stable, but also
 be robust, in the sense that it tolerates, at the very least, small
time-variations in the plant parameters and a small amount of unmodelled dynamics.
Of course, if the plant and controller are both linear and time-invariant,
then such robustness follows from closed-loop stability---see \cite{zames}, \cite{desoer}.
On the other hand, if either the plant or controller is nonlinear,
this is often not the case and/or it is not easy to prove.

Recently it has been proven, in both the pole placement and first order
one-step-ahead settings, that if discrete-time
adaptive control is carried out in just the right way, then a (stable) convolution
bound can be obtained on the closed-loop behavior---see \cite{scl17} and \cite{pole18}; hence, the closed-loop system acts `linear-like', and
the convolution bound can be leveraged in a modular fashion\footnote{
  It is {\em modular} in the sense that we are able to leverage the results
  for the ideal case without reopening its proof;
  robustness can be proven directly from the convolution bound.
}
to prove that tolerance to
small time-variations and a small amount of unmodelled dynamics follows.
Of course, there
the controller is nonlinear and the nominal
plant is single-input, single-output, and LTI.
The goal of this paper is to generalize this result to
a larger class of multi-input multi-output plants and controllers.

To this end, here we consider a class of
finite-dimensional, nonlinear
plant and controller combinations;
if a convolution bound holds, then we prove that
tolerance to small time-variations in the plant parameters and a small amount of unmodelled dynamics follows.
An immediate application of this result is to prove robustness of our recently designed multi-estimator switching adaptive controllers presented in \cite{pole18} and \cite{cdc18}.
This result should also prove useful in extending our work on the adaptive control of LTI 
plants \cite{ccta17}, \cite{pole18}, \cite{cdc18}, \cite{acc19}
to that of nonlinear plants,
allowing us to focus on the ideal plant model in our analysis,
knowing that robustness will come for free.
Last of all, this result has the potential for use in other non-adaptive
(but nonlinear) contexts.

We denote ${\Z}$, ${\Z}^+$ and ${\N}$ as the sets of integers, non-negative integers and natural numbers, respectively. We will denote the Euclidean-norm of a vector and the induced norm of a matrix by the subscript-less default notation $\|\cdot\|$.
We let ${\mathbb S}(\R^{p\times q})$ denote the set of $\R^{p\times q} $-valued sequences.
We also let $\boldsymbol\ell_\infty(\R^{p\times q})$ denote the set of $\R^{p\times q} $-valued {\it bounded} sequences. If $\Omega\subset\R^{p\times q} $ is a bounded set, we define $\|\Omega\|:=\sup_{x\in\Omega}\|x\|$. 

Throughout this paper, we say that a function $\Gamma : \R^p \rightarrow \R^{q} $ has a {\em bounded gain} if  
there exists a
 $\nu>0$
 such that
 for all $x\in\R^{p}$, we have
 $\left\|\Gamma(x) \right\|
  \leq
  \nu
  \|x\|$;
the smallest such $\nu $ is the gain, and is denoted by $\|\Gamma\| $.

For a closed and convex set $\Omega\subset\R^p$, the function $\text{Proj}_\Omega \{\cdot\} : \R^p\rightarrow\Omega $
denotes the projection onto $\Omega$; it is well known that the function $\text{Proj}_\Omega $ is well defined.

\section{The Setup}

Here the nominal plant is multi-input multi-output
with finite memory and an additive disturbance,
such that the uncertain plant parameter enters linearly. 
To this end, with an output $y(t)\in\R^r$, an input $u(t)\in\R^m$, a disturbance $w(t)\in\R^r $,
 a modeling parameter of 
 $$\theta^*\in{\cal S}\subset\R^{p \times r},$$
and a vector of input-output data of the form
\[
  \phi(t)
  =
  \begin{bmatrix}
  y(t) \\ y(t-1) \\ \vdots \\ y(t-n_y+1) 
  \\
  u(t) \\ u(t-1) \\ \vdots \\ u(t-n_u+1)
  \end{bmatrix}
  \in \R^{n_y\cdot r+n_u\cdot m},
\]
we consider the plant
\begin{eqnarray}
  y(t+1)
  &=&
  {\theta^*}^\top
  f\bigl(\phi(t)\bigr) 
  + w(t),
  \quad
  \phi(t_0)=\phi_0;
   \label{plant1}
\end{eqnarray}
we assume that $f:\R^{n_y\cdot r+n_u\cdot m}\rightarrow\R^{p}$ 
has a \textbf{bounded gain} and that
 \textbf{${\cal S}$ is a bounded set};
both requirements are reasonable given that we will require 
uniform bounds in our analysis. 
We represent this system by the pair $\bigl(f,{\cal S}\bigr)$.

Here we consider a large class of controllers which
subsumes LTI ones as well as a large class of adaptive ones.
To this end, we consider a controller with its state partitioned into two parts:
\begin{itemize}
  \item 
  $z_1(t) \in \R^{l_1} $ and
  \item
  $z_2(t) \in\R^{l_2} $,
  \end{itemize}
an exogenous signal $r(t)\in\R^{q}$ (typically a reference signal), 
together with equations of the form
\begin{subequations}
\label{control}
\begin{flalign}
z_1(t+1) 
&=
g_1\left(
z_1(t), z_2(t), \phi(t), y(t+1),
r(t), t,t_0
\right),
\nonumber
\\
&\qquad\qquad\qquad\qquad\qquad\qquad\quad\;
z_1(t_0)=z_{1_0}
\label{control_3}
\\
z_2(t+1) 
&=
g_2\left(
z_1(t),z_2(t),\phi(t),y(t+1),
r(t), t,t_0
\right),
\nonumber
\\
&\qquad\qquad\qquad\qquad\qquad\qquad\quad\;
z_2(t_0)=z_{2_0} 
\label{control_1}
\\
u(t)
&=
h\left(
z_1(t),z_2(t),\phi(t),r(t)
\right). 
\label{control_2}
\end{flalign}
\end{subequations}
Here we assume that
\[
  g_2:
  \R^{l_1} \times {\cal X} \times 
  \R^{n_y\cdot r+n_u\cdot m}
  \times \R^r
  \times \R^q
  \times \Z \times \Z
  \longrightarrow
  {\cal X},
\]
i.e. if $z_2$ is initialized in ${\cal X} $, then it remains
in ${\cal X} $ throughout.

\begin{remark}
This class subsumes finite-dimensional LTI controllers: simply set $l_2=0$
so that the sub-state $z_2$ disappears, and make the functions $g_1 $ and $ h $ to be linear.
\end{remark}

\begin{remark}
This class subsumes many adaptive controllers:
simply set $l_1=0$ and let $z_2$ be the state of a parameter estimator constrained to the set ${\cal X} $.
\end{remark}

We now provide a definition of the desired linear-like closed-loop property:

\begin{mdframed}
\begin{definition}\label{def_conv}
We say that \eqref{control} {\bf provides a convolution bound for}
$\bigl(f,{\cal S}\bigr)$
with gain $c\geq1 $ and decay rate $\lambda\in(0,1) $ if, for every
$\theta^*\in{\cal S}$, $t_0\in\Z$, $\phi_0\in\R^{n_y\cdot r+n_u\cdot m}$, $z_{1_0}\in\R^{l_1} $, $z_{2_0}\in{\cal X}\subset\R^{l_2} $,  $w\in{\mathbb S}(\R^r) $ and $r\in{\mathbb S}(\R^q) $,
when \eqref{control} is applied to \eqref{plant1}, the following holds:
\begin{flalign}\label{conv_v1}
   &\left\|
   \begin{bmatrix}
   \phi(t)
   \\
   z_1(t)
   \end{bmatrix}
   \right\| 
  \leq
  c
  \lambda^{t-\tau}
 \left\|
   \begin{bmatrix}
   \phi(\tau)
   \\
   z_1(\tau)
   \end{bmatrix}
   \right\| 
  +
  \nonumber
  \\
  &\quad
  \sum_{j=\tau}^{t-1}
  c \lambda^{t-j-1}
   (\|r(j)\|+\|w(j)\|)
+c\|r(t)\|,
\nonumber
\\
  &\qquad\qquad\qquad
  t\geq\tau\geq t_0.
\end{flalign}
\end{definition}
\end{mdframed}

\begin{remark}
The reason why we do not focus on the exponential stability
aspect of \eqref{conv_v1} is that the 
$c
  \lambda^{t-\tau}
 \left\|
   \begin{bmatrix}
   \phi(\tau)
   \\
   z_1(\tau)
   \end{bmatrix}
   \right\| $
term can be viewed, in essence, as the effect of the past inputs on
the future, in much the same way as the `zero-input-response'
can be viewed in the analysis of LTI systems.
More specifically, the 
$c
  \lambda^{t-\tau}
 \left\|
   \begin{bmatrix}
   \phi(\tau)
   \\
   z_1(\tau)
   \end{bmatrix}
   \right\| $
term can be viewed as having arisen from a convolution
of the past inputs (before time $\tau$) with $c\lambda^t$, so
this term can be viewed as a convolution sum
in its own right.
\end{remark}

\section{Tolerance to Time-Variation}

We now consider plants with a possibly time-varying parameter vector $\theta^*(t)$ instead of a static $\theta^*$:
\begin{flalign}
\label{plantTV}
  y(t+1)
  &=
  \theta^*(t)^\top f\bigl(\phi(t)\bigr)
  +
  w (t),
  \quad
  \phi(t_0)=\phi_0.
\end{flalign} 
With $c_0\geq0$ and $\epsilon>0$,
let $s({\cal S},c_0,\epsilon)$ denote the subset of $\boldsymbol\ell_\infty(\R^{p\times r})$ whose elements $\theta^* $ satisfy: 
\begin{itemize}
\item
$\theta^*(t)\in{\cal S} $ for every $t\in\Z $, 
\item and
\[
  \sum_{t=t_1}^{t_2-1} 
  \Vert \theta^*(t+1)-\theta^*(t) \Vert
  \leq
  c_0 + \epsilon(t_2-t_1),\quad t_2>t_1,\; t_1\in\Z.
\]
\end{itemize}
The above time-variation model encompasses both slow variations and/or occasional jumps;
this class is well-known in the adaptive control literature, e.g. see \cite{kreiss}.
We can extend Definition \ref{def_conv} in a natural way to handle time-variations.

\begin{mdframed}
\begin{definition}\label{def_conv2}
We say that \eqref{control} {\bf provides a convolution bound for}
$\bigl(f,s({\cal S},c_0,\epsilon)\bigr)$ with gain $c\geq1 $ and decay rate $\lambda\in(0,1) $ if, for every
$\theta^*\in s({\cal S},c_0,\epsilon)$, $t_0\in\Z$, $\phi_0\in\R^{n_y\cdot r+n_u\cdot m}$, $z_{1_0}\in\R^{l_1} $, $z_{2_0}\in{\cal X}\subset\R^{l_2} $,  $w\in{\mathbb S}(\R^r) $ and $r\in{\mathbb S}(\R^q) $,
when \eqref{control} is applied to \eqref{plantTV}, the following holds:
\begin{flalign}\label{conv_v1b}
   &\left\|
   \begin{bmatrix}
   \phi(t)
   \\
   z_1(t)
   \end{bmatrix}
   \right\| 
  \leq
  c
  \lambda^{t-\tau}
 \left\|
   \begin{bmatrix}
   \phi(\tau)
   \\
   z_1(\tau)
   \end{bmatrix}
   \right\| 
  +
  \nonumber
  \\
  &\quad
  \sum_{j=\tau}^{t-1}
  c \lambda^{t-j-1}
   (\|r(j)\|+\|w(j)\|)
+c\|r(t)\|,
\nonumber
\\
  &\qquad\qquad\qquad
  t\geq\tau\geq t_0.
\end{flalign}
\end{definition}
\end{mdframed}

We now will show that if a controller \eqref{control} provides
convolution bounds for the plant \eqref{plant1}, then the same will
be true for the time-varying plant \eqref{plantTV},
as long as $\epsilon $ is small enough.
We consider two cases: one where there is a desired decay rate, and one where there is not.

\begin{mdframed}
\begin{theorem}\label{lemma_tv}
Suppose that the controller \eqref{control} provides 
a convolution bound for \eqref{plant1}
with gain $c\geq1 $ and decay rate $\lambda\in(0,1)$.
Then for every $\lambda_1\in(\lambda,1)$ and $c_0>0$, there exist a $c_1\geq c $ and $\epsilon>0 $ so that 
\eqref{control} provides a convolution bound for 
$\bigl(f,s({\cal S},c_0,\epsilon)\bigr)$
with gain $c_1 $ and decay rate $\lambda_1 $.
\end{theorem}
\end{mdframed}

\begin{remark}
This proof is based, in part, on the proof of Theorem 2 of \cite{pole18}, which deals with a much simpler setup.
\end{remark}

\begin{proof}[Proof of Theorem 1]
Suppose the controller \eqref{control} provides a convolution bound 
for \eqref{plant1} with gain $c\geq1 $ and a decay rate of $\lambda  $.
Fix $\lambda_1\in(\lambda,1) $ and $c_0>0 $; 
let
 $t_0\in\Z$, $\phi_0\in\R^{n_y\cdot r+n_u\cdot m}$, $z_{1_0}\in\R^{l_1} $, $z_{2_0}\in{\cal X} $, $w\in{\mathbb S}(\R^r) $ and $r\in{\mathbb S}(\R^q)$ be arbitrary. 

Now fix $m\in\N$ to be any number satisfying
\[
  m
  \geq
  \frac{\ln(c) 
  + 
  \frac{4 c_0 c\|f\|}{\lambda_1 - \lambda} 
  \left[\ln{\left(1+2c\|f\|\|{\cal S}\|\right)} 
  + \ln(2) 
  - \ln(\lambda + \lambda_1) \right]}
  {
  \ln(2\lambda_1) - \ln(\lambda+\lambda_1)
  },
\]
(the rationale for this choice will be more clear shortly),
and set
\[
\epsilon = \frac{c_0}{m^2};
\]
let $\theta^*\in s({\cal S},c_0,\epsilon)$ be arbitrary
and apply the controller \eqref{control} to the time-varying plant \eqref{plantTV}.
To proceed, we analyze the closed-loop system behavior on intervals of length $m$, which we further analyze in groups of $m^2 $.

To proceed, let $\bar t\geq t_0$ be arbitrary. Define a sequence $\{\bar t_i\}$ by
\[
  \bar t_i=\bar t+im, \qquad i\in\Z^+.
\]
We can rewrite the time-varying plant as
\begin{flalign*}
  y(t+1)
  &=
  \theta(\bar t_i)^\top f(\phi(t))
  + w(t) 
  + \underbrace{\left[
    \theta(t)-\theta(\bar t_i) 
    \right]^\top f(\phi(t)) 
    }_{=:\tilde n_i(t)},
    \\ 
    &\qquad
  t\in[\bar t_i,\bar t_{i+1}).
\end{flalign*}
On the interval $[\bar t_i,\bar t_{i+1}]$, we can regard the plant 
as time-invariant, but with an extra disturbance; so by 
hypothesis,
\begin{flalign}
  &
   \left\|
   \begin{bmatrix}
   \phi(t)
   \\
   z_1(t)
   \end{bmatrix}
   \right\| 
  \leq
  c 
  \lambda^{t-\bar t_i}
 \left\|
   \begin{bmatrix}
   \phi(\bar t_i)
   \\
   z_1(\bar t_i)
   \end{bmatrix}
   \right\| +
   \nonumber
   \\
  &\quad
  \sum_{j=\bar t_i}^{t-1}
  c \lambda^{t-j-1}
   (\|r(j)\|+\|w(j)\|+\|\tilde n_i(j)\|)
   +c\|r(t)\|,
   \nonumber
   \\
  &\qquad\qquad\qquad
  t\in[\bar t_i,\bar t_{i+1}],\; i\in\Z^+.
\end{flalign}
To analyze this difference inequality,
we first construct an associated difference equation:
\[
  \psi(t+1)=\lambda \psi(t) + \|r(t)\| + \|w(t)\| 
  + \|\tilde n_i(t)\|,\qquad 
  t\in[\bar t_i,\bar t_{i+1}),
\]
with an initial condition of
   \[
  \psi(\bar t_i)=
  \left\|
   \begin{bmatrix}
   \phi(\bar t_i)
   \\
   z_1(\bar t_i)
   \end{bmatrix}
   \right\|.
\]
Using the fact that $c\geq1 $, it is straightforward to prove that
\begin{equation}
\left\|
   \begin{bmatrix}
   \phi(t)
   \\
   z_1(t)
   \end{bmatrix}
   \right\| \leq c\psi(t)
   +c\|r(t)\| , \;\; 
   t\in[\bar t_i,\bar t_{i+1}].
   \label{start_eq}
\end{equation}
Now we analyze this equation for $i=0,1,\ldots,m-1$.\\

\noindent{\bf Case 1:} $\|\tilde n_i(t)\|\leq \frac{\lambda_1-\lambda}{2c}\|\phi(t) \|$ for all $t\in[\bar t_i, \bar t_{i+1})$.

Using the above bound \eqref{start_eq} and the fact that 
$\lambda_1-\lambda\in(0,1) $, we obtain
\begin{flalign}
\psi(t+1) 
&\leq
\lambda \psi(t) + \|r(t)\| + \|w(t)\| 
+ \|\tilde n_i(t)\| 
\nonumber
\\
&\leq
\lambda \psi(t) + \|r(t)\| + \|w(t)\| 
+ \tfrac{\lambda_1-\lambda}{2c}\|\phi(t) \|
\nonumber
\\
&\leq
\lambda \psi(t) + \|r(t)\| + \|w(t)\| 
+ \tfrac{\lambda_1-\lambda}{2}\left(\psi(t)+\|r(t)\|\right)
\nonumber
\\
&\leq
\tfrac{\lambda_1+\lambda}{2} \psi(t) 
+2 \|r(t)\| + \|w(t)\|, 
\quad t\in[\bar t_i, \bar t_{i+1}),
\end{flalign}
which means that
\begin{flalign}
|\psi(t)| 
&\leq
\left(
\tfrac{\lambda_1+\lambda}{2}
\right)^{t-\bar t_i}
|\psi(\bar t_i)|
+
\nonumber
\\
&
\sum_{j=\bar t_i}^{t-1}
\left(
\tfrac{\lambda_1+\lambda}{2}
\right)^{t-j-1}
\left(
2\|r(j) \| + \|w(j) \| 
\right), 
\nonumber \\
&
\qquad 
t=\bar t_i, \bar t_i+1,\ldots, \bar t_{i+1}.\label{tv1_bd1x}
\end{flalign}
This, in turn, implies that there exists $c_2\geq 2 c$ so that
\begin{flalign}
&\left\|
\begin{bmatrix}
\phi(\bar t_{i+1})
\\
z_1(\bar t_{i+1})
\end{bmatrix}
\right\| 
\leq
c\left(
\tfrac{\lambda_1+\lambda}{2}
\right)^{m}
\left\|
\begin{bmatrix}
\phi(\bar t_{i})
\\
z_1(\bar t_{i})
\end{bmatrix}
\right\| 
+
\nonumber
\\
&
\sum_{j=\bar t_i}^{\bar t_{i+1} -1}
c_2
\left(
\tfrac{\lambda_1+\lambda}{2}
\right)^{\bar t_{i+1}-j-1}
\left(
\|r(j) \| + \|w(j) \| 
\right)
+c_2\|r(\bar t_{i+1})\|
. \label{tv_bd1}
\end{flalign}

\noindent{\bf Case 2:} $\|\tilde n_i(t)\|> \frac{\lambda_1-\lambda}{2c}\|\phi(t) \|$ for some $t\in[\bar t_i, \bar t_{i+1})$.

Since $\theta^*(t)\in{\cal S}$ for $t\geq t_0$, we see that
\[
  \|\tilde n_i(t)\|
  \leq
  2\|{\cal S}\| \|f(\phi(t))\|
  \leq
  2 \|f\| \|{\cal S}\| \times \|\phi(t)\|, 
  \quad t\in[\bar t_i,\bar t_{i+1}).
\]
This means that 
\begin{flalign}
&
\psi(t+1) 
\leq
\lambda \psi(t) + \|r(t)\| + \|w(t)\| + \|\tilde n_i(t)\| 
\nonumber
\\
&
\leq
\lambda \psi(t) + \|r(t)\| + \|w(t)\| 
+ 2\|f\| \|{\cal S}\| \|\phi(t) \|
\nonumber
\\
&
\leq
\underbrace{\left(
1+2c\|f\|\|{\cal S}\|
\right)}_{=:\gamma_3}
 \psi(t) + 
 (1+2c\|f\|\|{\cal S}\|)\|r(t)\| +
 \|w(t)\|, 
   \nonumber
 \\
 & 
\qquad
 \qquad
  t\in[\bar t_i, \bar t_{i+1}),
\end{flalign}
which means that 
\begin{flalign}
|\psi(t)| 
&\leq
\gamma_3^{t-\bar t_i}
|\psi(\bar t_i)|
+
\sum_{j=\bar t_i}^{t-1}
\gamma_3^{t-j-1}
\left(
\gamma_3\|r(j) \| + \|w(j) \| 
\right), 
\nonumber \\
&
\qquad 
t=\bar t_i, \bar t_i+1,\ldots, \bar t_{i+1}.\label{tv_bd2x}
\end{flalign}
Setting $t=\bar t_{i+1} $ and using \eqref{start_eq} yields
\begin{flalign}
&\left\|
\begin{bmatrix}
\phi(\bar t_{i+1})
\\
z_1(\bar t_{i+1})
\end{bmatrix}
\right\| 
\leq
c\gamma_3^{m}
\left\|
\begin{bmatrix}
\phi(\bar t_{i})
\\
z_1(\bar t_{i})
\end{bmatrix}
\right\| 
+
\nonumber
\\
&
\sum_{j=\bar t_i}^{\bar t_{i+1} -1}
c
\gamma_3^{\bar t_{i+1}-j-1}
\left(
\gamma_3\|r(j) \| + \|w(j) \| 
\right)
+c\|r(\bar t_{i+1})\|
\nonumber \\
&\leq
c\gamma_3^{m}
\left\|
\begin{bmatrix}
\phi(\bar t_{i})
\\
z_1(\bar t_{i})
\end{bmatrix}
\right\| 
+
c
\gamma_3
\left(
\tfrac{2\gamma_3}{\lambda_1+\lambda}
\right)^m
\times
\nonumber
\\
&
\sum_{j=\bar t_i}^{\bar t_{i+1} -1}
\left(
\tfrac{\lambda_1+\lambda}{2}
\right)^{\bar t_{i+1}-j-1}
\left(
\|r(j) \| + \|w(j) \| 
\right)
+c\|r(\bar t_{i+1})\|
. \label{tv_bd2}
\end{flalign}
This completes Case 2.

{At this point we combine Case 1 and 2.}
We would like to analyze $m $ intervals of length $m $.
On the
interval $[\bar t,\bar t+ m^2]$,
 there are $m$ subintervals of length $m$; furthermore, because of the choice of $\epsilon$ we have that
\[
  \sum_{j=\bar t}^{\bar t+m^2-1}
  \|\theta(j+1)-\theta(j)\|
  \leq
  c_0 + \epsilon m^2
  \leq
  2c_0.
\]
It is easy to see that there are at most
 $N_1:=\frac{4c_0 c\|f\|}{\lambda_1-\lambda}$ 
 subintervals which fall into the category of Case 2, with the remainder falling into the category of Case 1; it is clear from the formula for $m$ that $m>N_1$. If we use \eqref{tv_bd1} and \eqref{tv_bd2} to analyze the behavior of the closed-loop system on the interval $[\bar t,\bar t+m^2]$, we end up with a crude bound of 
\begin{flalign}
&
\left\|
\begin{bmatrix}
\phi(\bar t+m^2)
\\
z_1(\bar t+m^2)
\end{bmatrix}
\right\| 
\leq
c^m\gamma_3^{N_1m}
\left(
\tfrac{\lambda_1+\lambda}{2}
\right)^{m(m-N_1)}
\left\|
\begin{bmatrix}
\phi(\bar t)
\\
z_1(\bar t)
\end{bmatrix}
\right\| +
\nonumber 
\\
&\qquad
2
m
\left(
\tfrac{2\gamma_3}{\lambda_1+\lambda}
\right)^m
(c_2\gamma_3^{m+1})^m
\left(
\tfrac{2}{\lambda_1+\lambda}
\right)^{(m+1)m}
\times
\nonumber
\\
&
\sum_{j=\bar t}^{\bar t+m^2-1}
\left(
\tfrac{\lambda_1+\lambda}{2}
\right)^{\bar t+m^2-j-1}
\left(
\|r(j) \| + \|w(j) \| 
\right)
+
\nonumber
\\
&\qquad
c_2\|r(\bar t+m^2)\|
. \label{tv_bd3}
\end{flalign}
From the choice of $m$ above, 
it is easy to show that
\[
  m^2 \ln\left(
   \tfrac{2\lambda_1}{\lambda_1+\lambda}
  \right)
  \geq
  m\ln(c) + N_1m\ln(\gamma_3)
  +N_1m\ln\left(\tfrac{2}{\lambda+\lambda_1}\right);
\]
this immediately implies that
\[
  c^m \gamma_3^{N_1m}\left(\frac{2}{\lambda+\lambda_1}\right)^{N_1m}
  \leq
  \left(
   \frac{2\lambda_1}{\lambda_1+\lambda}
  \right)^{m^2}
  \quad
\]
\[
  \Leftrightarrow
  \quad
  c^m \gamma_1^{N_1m} 
  \left(
\frac{\lambda_1+\lambda}{2}
\right)^{m(m-N_1)}
\leq
\lambda_1^{m^2}.
\]
Since $\frac{\lambda_1+\lambda}{2}<\lambda_1$, it follows from \eqref{tv_bd3} that there exists a constant $\gamma_4$ so that
\begin{flalign}
 &\left\|
\begin{bmatrix}
\phi(\bar t+m^2)
\\
z_1(\bar t+m^2)
\end{bmatrix}
\right\|  
\leq
\lambda_1^{m^2}
 \left\|
\begin{bmatrix}
\phi(\bar t)
\\
z_1(\bar t)
\end{bmatrix}
\right\| 
+
\nonumber
\\
&
\gamma_4
\sum_{j=\bar t}^{\bar t+m^2 -1}
\lambda_1^{\bar t+m^2-j-1}
\left(
\|r(j) \| + \|w(j) \| 
\right)
+\gamma_4\|r(\bar t+m^2)\|.
\label{tv_bd_v3}
\end{flalign}
Now let $\tau\geq t_0$ be arbitrary. By setting 
$\bar t=\tau, \tau+m^2, \tau+2m^2, \ldots $,
in succession, it follows from \eqref{tv_bd_v3} that
\begin{flalign}
&\left\|
\begin{bmatrix}
\phi(\tau+qm^2)
\\
z_1(\tau+qm^2)
\end{bmatrix}
\right\|  
\leq
\lambda_1^{qm^2}
 \left\|
\begin{bmatrix}
\phi(\tau)
\\
z_1(\tau)
\end{bmatrix}
\right\| 
+
\nonumber
\\
&\qquad
\gamma_4
\sum_{j=\tau}^{\tau+qm^2 -1}
\lambda_1^{\tau+qm^2-j-1}
\left(
\|r(j) \| + \|w(j) \| 
\right)
\nonumber
\\
&\qquad\quad
+\gamma_4
\|r(\tau+qm^2)\|
,\qquad q\in\Z^+. \label{tv_bd4}
\end{flalign}
So
 $\begin{bmatrix}
\phi(t)
\\
z_1(t)
\end{bmatrix}$
is well-behaved at 
$t=\tau, \tau+m^2, \tau+2m^2$, etc;
we can use \eqref{tv1_bd1x} of Case 1, \eqref{tv_bd2x} of Case 2 and 
\eqref{start_eq} to prove that nothing untoward happens between these times.
We conclude that there exists a constant $\gamma_5$ so that
\begin{flalign}
 & \left\|
\begin{bmatrix}
\phi(t)
\\
z_1(t)
\end{bmatrix}
\right\|  
\leq
\gamma_5
\lambda_1^{t-\tau}
\left\|
\begin{bmatrix}
\phi(\tau)
\\
z_1(\tau)
\end{bmatrix}
\right\|  
+
\nonumber
\\
&
\gamma_5
\sum_{j=\tau}^{t-1}\lambda_1^{t-j-1}
\left(
\|r(j) \| + \|w(j) \| 
\right)
+\gamma_5\|r(t)\|
,\quad t\geq \tau.
\end{flalign}
Since $\tau\geq t_0$ is arbitrary, the desired bound is proven.
\end{proof}

A careful examination of the above proof reveals that $\epsilon\rightarrow0 $ as $c_0\rightarrow0 $ and as
$c_0\rightarrow\infty $.
If we do not care about the decay rate, then we can remove this drawback.

\begin{mdframed}
\begin{theorem}\label{lemma_tv_v2}
Suppose that the controller \eqref{control} provides a convolution 
bound for \eqref{plant1} with gain $c\geq1 $ and decay rate $\lambda\in(0,1) $. Then there exists an $\epsilon>0 $ such that for every $c_0\geq0$, there exist $\lambda_* \in(0,1) $ and $\gamma>0 $ so that 
\eqref{control} provides a convolution bound for 
$\bigl(f,s({\cal S},c_0,\epsilon)\bigr)$
with gain $\gamma $ and decay rate $\lambda_* $.
\end{theorem}
\end{mdframed}

\begin{proof}[Proof of Theorem \ref{lemma_tv_v2}]

Suppose the controller \eqref{control} provides a convolution bound
for \eqref{plant1} with gain $c\geq1 $ and a decay rate of $\lambda $.
Fix $\lambda_1\in(\lambda,1) $; let $t_0\in\Z$, $\phi_0\in\R^{n_y\cdot r+n_u\cdot m}$, $z_{1_0}\in\R^{l_1} $, $z_{2_0}\in{\cal X} $, $w\in{\mathbb S}(\R^{r}) $ and $r\in{\mathbb S}(\R^{q})$ be arbitrary. The goal is to prove that for a small-enough $\epsilon$, the controller \eqref{control} provides a convolution bound for
$\bigl(f,s({\cal S},c_0,\epsilon)\bigr)$ for every $c_0\geq0$.
So at this point we will analyze the closed-loop system for an arbitrary $\epsilon>0 $, $c_0\geq0 $, and
$\theta^*\in s({\cal S},c_0,\epsilon)$.

To proceed, let $\bar t\geq t_0$ be arbitrary. For $m\in\N$, we will first analyze closed-loop behavior on intervals of length $m$; define a sequence $\{\bar t_i\}$ by
\[
  \bar t_i=\bar t+im, \qquad i\in\Z^+.
\]
We can rewrite the time-varying plant as
\begin{flalign*}
  y(t+1)
  &=
  \theta(\bar t_i)^\top f(\phi(t))
  + w(t) 
  + \underbrace{\left[
    \theta(t)-\theta(\bar t_i) 
    \right]^\top f(\phi(t)) 
    }_{=:\tilde n_i(t)}, 
    \nonumber
    \\
    &\qquad
  t\in[\bar t_i,\bar t_{i+1}).
\end{flalign*}
On the interval $[\bar t_i,\bar t_{i+1}]$, we regard the plant as time-invariant, but with an extra disturbance: so we obtain
\begin{flalign}
  & \left\|
   \begin{bmatrix}
   \phi(t)
   \\
   z_1(t)
   \end{bmatrix}
   \right\| 
  \leq
  c 
  \lambda^{t-\bar t_i}
 \left\|
   \begin{bmatrix}
   \phi(\bar t_i)
   \\
   z_1(\bar t_i)
   \end{bmatrix}
   \right\| 
  +
  \nonumber
  \\
  &
  \sum_{j=\bar t_i}^{t-1}
  c \lambda^{t-j-1}
     (\|r(j)\|+\|w(j)\|
     +
   \|\tilde n_i(j)\|
   )
      +c\|r(t)\|
   ,
   \nonumber
   \\
   &\qquad
  \qquad
  t\in[\bar t_i,\bar t_{i+1}],\; i\in\Z^+.
\end{flalign}

Using the same idea as in the proof of Theorem 1, we define the difference equation
\[
  \psi(t+1)=\lambda \psi(t) 
  + \|r(t)\| + \|w(t)\| 
  + \|\tilde n_i(t)\|,\qquad 
  t\in[\bar t_i,\bar t_{i+1})
\]
with
   \[
  \psi(\bar t_i)=
  \left\|
   \begin{bmatrix}
   \phi(\bar t_i)
   \\
   z_1(\bar t_i)
   \end{bmatrix}
   \right\|;
\]
it follows that
\begin{equation}
\left\|
   \begin{bmatrix}
   \phi(t)
   \\
   z_1(t)
   \end{bmatrix}
   \right\| \leq c\psi(t)
   +c\|r(t)\| 
   , \;\; 
   t\in[\bar t_i,\bar t_{i+1}].
   \label{star_eq2}
\end{equation}

\noindent{\bf Case 1:} $\|\tilde n_i(t)\|\leq \frac{\lambda_1-\lambda}{2c}\|\phi(t) \|$ for all $t\in[\bar t_i, \bar t_{i+1})$.

Arguing in an identical manner to the proof of Theorem 1, we obtain the following two bounds:
\begin{flalign}
|\psi(t)| 
&\leq
\left(
\tfrac{\lambda_1+\lambda}{2}
\right)^{t-\bar t_i}
|\psi(\bar t_i)|
+
\nonumber
\\
&
\sum_{j=\bar t_i}^{t-1}
\left(
\tfrac{\lambda_1+\lambda}{2}
\right)^{t-j-1}
\left(
2\|r(j) \| + \|w(j) \| 
\right), 
\nonumber \\
&
\qquad 
t=\bar t_i, \bar t_i+1,\ldots, \bar t_{i+1};\label{tv_bd1x}
\end{flalign}
this, in turn, implies that there exists $c_2>c$ so that
\begin{flalign}
&\left\|
\begin{bmatrix}
\phi(\bar t_{i+1})
\\
z_1(\bar t_{i+1})
\end{bmatrix}
\right\| 
\leq
c\left(
\tfrac{\lambda_1+\lambda}{2}
\right)^{m}
\left\|
\begin{bmatrix}
\phi(\bar t_{i})
\\
z_1(\bar t_{i})
\end{bmatrix}
\right\| 
+
\nonumber
\\
&
\sum_{j=\bar t_i}^{\bar t_{i+1} -1}
c_2
\left(
\tfrac{\lambda_1+\lambda}{2}
\right)^{\bar t_{i+1}-j-1}
\left(
\|r(j) \| + \|w(j) \| 
\right)
+c_2\|r(\bar t_{i+1})\|
. \label{tv2_bd1}
\end{flalign}

\noindent{\bf Case 2:} $\|\tilde n_i(t)\|> \frac{\lambda_1-\lambda}{2c}\|\phi(t) \|$ for some $t\in[\bar t_i, \bar t_{i+1})$.

Arguing in an identical manner to the proof of Theorem 1, we obtain the following two bounds: there exists $\gamma_3>0$ so that
\begin{flalign}
|\psi(t)| 
&\leq
\gamma_3^{t-\bar t_i}
|\psi(\bar t_i)|
+
\sum_{j=\bar t_i}^{t-1}
\gamma_3^{t-j-1}
\left(
\gamma_3\|r(j) \| + \|w(j) \| 
\right), 
\nonumber \\
&
\qquad 
t=\bar t_i, \bar t_i+1,\ldots, \bar t_{i+1},\label{tv2_bd2x}
\end{flalign}
\begin{flalign}
&
\left\|
\begin{bmatrix}
\phi(\bar t_{i+1})
\\
z_1(\bar t_{i+1})
\end{bmatrix}
\right\| 
\leq
c\gamma_3^{m}
\left\|
\begin{bmatrix}
\phi(\bar t_{i})
\\
z_1(\bar t_{i})
\end{bmatrix}
\right\| 
+
\nonumber
\\
&
\sum_{j=\bar t_i}^{\bar t_{i+1} -1}
c
\gamma_3^{\bar t_{i+1}-j-1}
\left(
\gamma_3\|r(j) \| + \|w(j) \| 
\right)
+c\|r(\bar t_{i+1})\|
\nonumber \\
&\leq
c\gamma_3^{m}
\left\|
\begin{bmatrix}
\phi(\bar t_{i})
\\
z_1(\bar t_{i})
\end{bmatrix}
\right\| 
+
c_2
\gamma_3
\left(
\tfrac{2\gamma_3}{\lambda_1+\lambda}
\right)^m
\times
\nonumber
\\
&
\sum_{j=\bar t_i}^{\bar t_{i+1} -1}
\left(
\tfrac{\lambda_1+\lambda}{2}
\right)^{\bar t_{i+1}-j-1}
\left(
\|r(j) \| + \|w(j) \| 
\right)
+c_2\|r(\bar t_{i+1})\|
. \label{tv2_bd2}
\end{flalign}
This completes Case 2.

{At this point we combine Case 1 and 2.} We would like to analyze $\bar N\in\N $ intervals of length $m$; for now we let $\bar N $ be free. We see that
\[
  \sum_{j=\bar t}^{\bar t+m\bar N-1}
  \|\theta(j+1)-\theta(j)\|
  \leq
  c_0 + \epsilon m\bar N.
\]
Let $N_1 $ denote the number of intervals of the form $[\bar t_i,\bar t_{i+1})$ which lie in $[\bar t,\bar t+m\bar N]$ which fall into Case 2; it is easy to see that $N_1 $ satisfies
\begin{flalign}
 &\qquad
  N_1
  \times
  \frac{\lambda_1-\lambda}{2c}
  \leq
  \left(c_0 + \epsilon m\bar N\right)\|f\|
  \nonumber
\\
\Rightarrow&
  N_1 
  \leq
    \left(\frac{2c\|f\|}{\lambda_1-\lambda}\right)
\times
  c_0 
  + 
\left(  \frac{2c\|f\|}{\lambda_1-\lambda}
    \right)  \times
  \epsilon \times m\bar N;
  \label{N1_bdx1}
\end{flalign}
observe that $N_1 $ depends on both $c_0 $ and $\epsilon $.
Using 
\eqref{tv2_bd1} and \eqref{tv2_bd2} we obtain
\begin{flalign}
  &
  \left\|
\begin{bmatrix}
\phi(\bar t+m\bar N)
\\
z_1(\bar t+m\bar N)
\end{bmatrix}
\right\| 
\leq
c^{\bar N}
\left(
\tfrac{\lambda_1+\lambda}{2}
\right)^{m(\bar N-N_1)}
\gamma_3^{mN_1}
\left\|
\begin{bmatrix}
\phi(\bar t)
\\
z_1(\bar t)
\end{bmatrix}
\right\|
+
\nonumber 
\\
&\qquad
2
\bar N
\left(
\tfrac{2\gamma_3}{\lambda_1+\lambda}
\right)^{\bar N}
(c_2\gamma_3^{m+1})^{\bar N}
\left(
\tfrac{2}{\lambda_1+\lambda}
\right)^{(m+1)\bar N}
\times
\nonumber
\\
&
\sum_{j=\bar t}^{\bar t+m\bar N -1}
\left(
\tfrac{\lambda_1+\lambda}{2}
\right)^{\bar t+m\bar N-j-1}
\left(
\|r(j) \| + \|w(j) \| 
\right)
\nonumber
\\
&\qquad
+c_2\|r(\bar t_{(q+1)m})\|
. 
\label{N1_main}
\end{flalign}
At this point, we will choose quantities $m,\epsilon$ and $\bar N $, in that order, so that the key gain 
$c^{\bar N}
\left(
\tfrac{\lambda_1+\lambda}{2}
\right)^{m(\bar N-N_1)}
\gamma_3^{mN_1}
<1 $.
First of all, we apply the bound on $N_1 $ given in \eqref{N1_bdx1}
to this key gain:
\begin{flalign}
&
c^{\bar N}
\left(
\tfrac{\lambda_1+\lambda}{2}
\right)^{m(\bar N-N_1)}
\gamma_3^{mN_1}
=
  \left[
  c
    \left(
      \tfrac{\lambda_1+\lambda}{2}
    \right)^{m}
    \right]^{\bar N}
\left[
\left(
\tfrac{2\gamma_3}{\lambda_1+\lambda}
\right)^{N_1}
\right]^m
\nonumber
\\
&\leq
    \left[
  c
    \left(
      \tfrac{\lambda_1+\lambda}{2}
    \right)^{m}
    \right]^{\bar N}
\left[
\left(
\tfrac{2\gamma_3}{\lambda_1+\lambda}
\right)^{
  \left[
        \left(\frac{2c\|f\|}{\lambda_1-\lambda}\right)
      c_0 
      + 
    \left(  \frac{2c\|f\|}{\lambda_1-\lambda}
        \right)  
      \epsilon  m\bar N\right]
}
\right]^m.
\label{N1_1}
\end{flalign}
{\bf Now choose} $m$ so that 
$
  c
        \left(
          \frac{\lambda_1+\lambda}{2}
        \right)^{m}
    =:\lambda_2
    <
    1
$,
i.e. any
$
  m
  >
  \frac{\ln(c)
  }{\ln(2)-\ln(\lambda_1+\lambda)}
$.
So rewriting \eqref{N1_1}, we now obtain
\begin{flalign*}
&
c^{\bar N}
\left(
\tfrac{\lambda_1+\lambda}{2}
\right)^{m(\bar N-N_1)}
\gamma_3^{mN_1}
\leq
\nonumber
\\
&
    \left[
  \left(
\tfrac{2\gamma_3}{\lambda_1+\lambda}
\right)^{
  \left(\frac{2c\|f\|}{\lambda_1-\lambda}\right)
    \times
      c_0 m
}
\right]
  \left( 
        \left[\left(
                    \tfrac{2\gamma_3}{\lambda_1+\lambda}
                    \right)^{
                      \left(  \frac{2c\|f\|}{\lambda_1-\lambda}
                            \right)  \times
                          \epsilon \times m^2
                    } \right]
    \times
    \lambda_2
    \right)^{ \bar N}.
\end{flalign*}
Now observe that
\[
\lim_{\epsilon \rightarrow 0 }
\left[
\left(
                    \tfrac{2\gamma_3}{\lambda_1+\lambda}
                    \right)^{
                      \left(  \frac{2c\|f\|}{\lambda_1-\lambda}
                            \right)  \times
                          \epsilon \times m^2
                    }
                    \right]
    =1,
\]
so now choose $\epsilon>0 $ so that 
\[
   \underbrace{\left[\left(
                    \tfrac{2\gamma_3}{\lambda_1+\lambda}
                    \right)^{
                      \left(  \frac{2c\|f\|}{\lambda_1-\lambda}
                            \right)  \times
                          \epsilon \times m^2
                    } \right]
    \times
    \lambda_2}_{
    =:
    \lambda_{3}
    }
    <1;
\]
notice that $\epsilon $ is independent of $c_0 $.
With this choice we now have
\[
  c^{\bar N}
\left(
\tfrac{\lambda_1+\lambda}{2}
\right)^{m(\bar N-N_1)}
\gamma_3^{mN_1}
\leq
 \left[
  \left(
\tfrac{2\gamma_3}{\lambda_1+\lambda}
\right)^{
  \left(\frac{2c\|f\|}{\lambda_1-\lambda}\right)
    \times
      c_0 m
}
\right]
\times
\lambda_3^{\bar N}.
\]
Last of all, now choose $\bar N $ so that
$$
     \underbrace{\left[
  \left(
\tfrac{2\gamma_3}{\lambda_1+\lambda}
\right)^{
  \left(\frac{2c\|f\|}{\lambda_1-\lambda}\right)
    \times
      c_0 m
}
\right]
\times
\lambda_3^{\bar N}}_{=:\lambda_4}
<
1;
$$
any
$
  \bar N
  >
  \frac{2c c_0 m\|f\|
  \left[
\ln(2\gamma_3)-\ln(\lambda_1+\lambda)
  \right] }{
  (\lambda-\lambda_1)\ln(\lambda_3)
  }
$
will do. Observe that $\bar N $ depends on $c_0 $.

So incorporating all of the above, there exists $\gamma_{4}>0$ 
(which clearly depends on $c_0 $ via $\bar N $)
so that we can rewrite \eqref{N1_main} as
\begin{flalign}
  &
  \left\|
\begin{bmatrix}
\phi(\bar t+m\bar N)
\\
z_1(\bar t+m\bar N)
\end{bmatrix}
\right\| 
\leq
\lambda_4
\left\|
\begin{bmatrix}
\phi(\bar t)
\\
z_1(\bar t)
\end{bmatrix}
\right\|
+
\nonumber 
\\
&\qquad
\gamma_4
\sum_{j=\bar t}^{\bar t+m\bar N -1}
\left(
\tfrac{\lambda_1+\lambda}{2}
\right)^{\bar t+m\bar N-j-1} 
\left(
\|r(j) \| + \|w(j) \| 
\right)
+
\nonumber
\\
&\qquad\quad
\gamma_4\|r(\bar t+m\bar N)\|
.
\label{N1_main2}
\end{flalign}
Now let $\tau\geq t_0 $ be arbitrary. By setting $\bar t=\tau,\tau+m\bar N, \tau+2m\bar N,\ldots $, in succession, with 
$\lambda_5:=\max\left\{
\lambda_4^{\tfrac{1}{m\bar N}}, \tfrac{\lambda_1+\lambda}{2}
\right\} $ 
(which clearly depends on $c_0 $ via $\bar N $)
it follows from \eqref{N1_main2} that
\begin{flalign}
&\left\|
\begin{bmatrix}
\phi(\bar t+q\bar Nm)
\\
z_1(\bar t+q\bar Nm)
\end{bmatrix}
\right\|  
\leq
\lambda_5^{q\bar Nm}
 \left\|
\begin{bmatrix}
\phi(\bar t)
\\
z_1(\bar t)
\end{bmatrix}
\right\| 
+
\nonumber
\\
&
\gamma_4
\sum_{j=\bar t}^{\bar t+q\bar Nm -1}
\lambda_5^{\bar t+q\bar Nm-j-1}
\left(
\|r(j) \| + \|w(j) \| 
\right)
+
\nonumber
\\
&\qquad\qquad
\gamma_4
\|r(\bar t+q\bar Nm)\|
,\qquad q\in\Z^+. \label{tv2_bd4}
\end{flalign}
So
 $\begin{bmatrix}
\phi(t)
\\
z_1(t)
\end{bmatrix}$
is well-behaved at 
$t=\tau, \tau+m^2, \tau+2m^2$, etc;
we can use \eqref{tv_bd1x} of Case 1, \eqref{tv2_bd2x} of Case 2 and 
\eqref{star_eq2} to prove that nothing untoward happens between these times.
We conclude that there exists a constant $\gamma_5$ so that
\begin{flalign}
 & \left\|
\begin{bmatrix}
\phi(t)
\\
z_1(t)
\end{bmatrix}
\right\|  
\leq
\gamma_5
\lambda_5^{t-\tau}
\left\|
\begin{bmatrix}
\phi(\tau)
\\
z_1(\tau)
\end{bmatrix}
\right\|  
+
\nonumber
\\
&
\gamma_5
\sum_{j=\tau}^{t-1}\lambda_5^{t-j-1}
\left(
\|r(j) \| + \|w(j) \| 
\right)
+\gamma_5\|r(t)\|
,\quad t\geq \tau.
\end{flalign}
Since $\tau\geq t_0$ is arbitrary, the desired bound is proven.
\end{proof}

\section{Tolerance to Unmodelled Dynamics}

We now consider the time-varying plant \eqref{plantTV} with the term
$d_\Delta(t)\in\R^r $ added to represent unmodelled dynamics:
\begin{flalign}
\label{plant_umd}
  y(t+1)
  &=
  \theta^*(t)^\top f\bigl(\phi(t)\bigr) 
  +
  w (t)
  +
  d_\Delta(t),\quad
  \phi(t_0)=\phi_0.
\end{flalign}
Here we consider (a generalized version of) a class of
unmodelled dynamics which is common in the adaptive control literature---see
\cite{kreiss86} and \cite{pole18}.
With 
$g:\R^{n_y\cdot r+n_u\cdot m} \rightarrow \R $
a map with a bounded gain, $\beta\in(0,1) $ and $\mu>0 $, we consider
\begin{subequations}
\label{d_deltaX}
\begin{flalign}
m(t+1)
&=
\beta m(t) +\beta |g(\phi(t))|, 
\quad 
m(t_0)=m_0
\label{d_delta1}
\\
\|d_\Delta(t)\|
&\leq
\mu m(t)+\mu |g(\phi(t))|,\quad t\geq t_0. \label{d_delta2}
\end{flalign}
\end{subequations}
It turns out that this model subsumes 
{a large class of}
classical additive uncertainty, multiplicative uncertainty, and uncertainty in a coprime factorization, with 
a
strict causality constraint;
see \cite{pole18} for a more detailed explanation.
We will now show that if the controller \eqref{control} provides a convolution bound for 
$\bigl(f,s({\cal S},c_0,\epsilon)\bigr)$,
then a degree of tolerance to unmodelled dynamics can be proven.

\begin{mdframed}
\begin{theorem}\label{umd_thm}
Suppose that the controller \eqref{control} provides a convolution bound for $\bigl(f,s({\cal S},c_0,\epsilon)\bigr)$ 
with a gain $c_1 $ and decay rate $\lambda_1\in(0,1) $. 
Then for every $\beta\in(0,1) $ and $\lambda_2\in(\max\{\lambda_1,\beta\},1)$, there exist $\bar\mu>0$ and $c_2>0 $ so that for every $\theta^*\in s({\cal S},c_0,\epsilon)$, $\mu\in(0,\bar \mu) $,
 $t_0\in\Z$, $\phi_0\in\R^{n_y\cdot r+n_u\cdot m}$,  $z_{1_0}\in\R^{l_1} $, $z_{2_0}\in{\cal X}\subset\R^{l_2} $, $r \in {\mathbb S}(\R^q)$ and $w \in {\mathbb S}(\R^r) $, when the controller \eqref{control} is applied to the plant \eqref{plant_umd} with $d_\Delta$ satisfying \eqref{d_deltaX}, the following holds:
\begin{flalign}
&   \left\|
   \begin{bmatrix}
   \phi(t)
   \\
   z_1(t)
   \\
   m(t)
   \end{bmatrix}
   \right\| 
  \leq
  c_2 
  \lambda_2^{t-t_0}
 \left\|
       \begin{bmatrix}
       \phi_0
       \\
       z_{1_0}
       \\
       m_0
       \end{bmatrix}
\right\|     
  +
  \nonumber
  \\
  &
  \sum_{j=t_0}^{t-1}
  c_2 \lambda_2^{t-j-1}
   (\|r(j)\|+\|w(j)\|)
   +c_2\|r(t)\|
   ,
  \quad
  t\geq t_0.
\end{flalign}
\end{theorem}
\end{mdframed}

\begin{remark}
This proof is based, in part, on the proof of Theorem 3 of \cite{pole18},
which deals with a much simpler setup.
\end{remark}

\begin{proof}[Proof of Theorem \ref{umd_thm}]
Fix $\beta\in(0,1)$ and $\lambda_2\in(\max\{\lambda_1,\beta\},1)$ and let
$\theta^*\in s({\cal S},c_0,\epsilon)$,
$t_0\in\Z$, $\phi_0\in\R^{n_y\cdot r+n_u\cdot m}$, $z_{1_0}\in\R^{l_1} $, $z_{2_0}\in{\cal X} $, $w\in{\mathbb S}(\R^{r}) $ and $r\in{\mathbb S}(\R^{q})$ be arbitrary. So by hypothesis:
\begin{flalign}
   &\left\|
   \begin{bmatrix}
   \phi(t)
   \\
   z_1(t)
   \end{bmatrix}
   \right\| 
  \leq
  c_1
  \lambda_1^{t-\tau}
 \left\|
   \begin{bmatrix}
   \phi(\tau)
   \\
   z_1(\tau)
   \end{bmatrix}
   \right\| 
  +
  \nonumber
  \\
  &\quad
  \sum_{j=\tau}^{t-1}
  c_1 \lambda_1^{t-j-1}
   (\|r(j)\|+\|w(j)\|+\|d_\Delta(j)\| )
+c_1\|r(t)\|,
\nonumber
\\
  &\qquad\qquad\qquad
  t\geq\tau\geq t_0.
  \label{umd_bd4}
\end{flalign}
To convert this inequality to an equality, we consider the 
associated difference equations
\[
  \tilde\phi(t+1)
  =
  \lambda_1 \tilde\phi(t)
  + c_1 \|r(t)\| 
  + c_1 \|w(t) \| 
  + c_1\mu \tilde m(t) 
  + c_1\mu \|g\| \tilde \phi(t),
  \]
  \[ 
   \tilde \phi(t_0)=c_1
  \left\| \begin{bmatrix}
       \phi_0
       \\
       z_{1_0}
       \end{bmatrix}
       \right\|,
\]
together with the difference equation based on \eqref{d_delta1}:
\[
  \tilde m(t+1) 
  =
  \beta \tilde m(t)
  + \beta \|g\|
   \tilde\phi(t), 
  \quad \tilde m(t_0)=|m_0|.
\]
Using induction together with \eqref{umd_bd4}, \eqref{d_delta1}, and \eqref{d_delta2}, we can prove that
\begin{subequations}
\label{umd_prv1}
\begin{flalign}
\left\| 
\begin{bmatrix}
       \phi(t)
       \\
       z_1(t)
       \end{bmatrix}
       \right\|
       &\leq
       \tilde \phi(t)
       +c_1\|r(t)\|
       ,
       \\
       \nonumber 
       \\
       |m(t)|
       &\leq \tilde m(t),
       \qquad\quad t\geq t_0. 
\end{flalign}
\end{subequations}
If we combine the difference equations for $\tilde \phi(t)
$ and $\tilde m(t)$, we obtain
\begin{flalign}
\begin{bmatrix}
\tilde\phi(t+1)
\\
\tilde m(t+1)
\end{bmatrix}
&=
\underbrace{
\begin{bmatrix}
\lambda_1+c_1\|g\| \mu & c_1\mu
\\
\beta \|g\| & \beta
\end{bmatrix}
}_{=: A_{\text{cl}}(\mu) }
\begin{bmatrix}
\tilde\phi(t)
\\
\tilde m(t)
\end{bmatrix}
+
\nonumber
\\
&\qquad
\begin{bmatrix}
c_1
\\
0
\end{bmatrix}
\left(
\|r(t) \| +\|w(t) \| 
\right),
\qquad 
t\geq t_0. \label{umd_prv2}
\end{flalign}
Now we see that 
\[
  A_{\text{cl}}(\mu) 
  \rightarrow
  \begin{bmatrix}
\lambda_1 & 0
\\
\beta \|g\| & \beta
\end{bmatrix}
\]
as $\mu\rightarrow0$, and this matrix has eigenvalues of $\{\lambda_1,\beta\} $
which are both less that $\lambda_2<1 $.
Using a standard Lyapunov argument, it is easy to prove that 
there exist
$\bar \mu>0 $ and $\gamma_1>0 $ such that for all $\mu\in(0,\bar\mu]$, we have
\[
  \left\|
A_{\text{cl}}(\mu)^k
  \right\|
  \leq   
\gamma_1
  \lambda_2^k,
  \qquad 
  k\geq0;
\]
if we use this in \eqref{umd_prv2} and then apply the bound in \eqref{umd_prv1}, it follows that 
\begin{flalign}
   & \left\|
   \begin{bmatrix}
   \phi(t)
   \\
   z_1(t)
   \\
    m(t)
   \end{bmatrix}
   \right\| 
   \leq
   c_1\gamma_1\lambda_2^{t-t_0}
    \left\|
       \begin{bmatrix}
       \phi_0
       \\
       z_{1_0}
       \\
       m_0
       \end{bmatrix}
\right\| 
+
\nonumber
\\
&\qquad
\sum_{j=t_0}^{t-1}
c_1\gamma_1\lambda_2^{t-j-1}
\left( 
\|r(j) \| +\|w(j) \| 
\right)
+c_1\|r(t)\|
,
\quad t\geq t_0
\end{flalign}
as desired.
\end{proof}

\section{Applications}

In this section, we will apply Theorems 1--3 to various adaptive control problems. In these examples, it turns out that we do not need $z_1 $ as part of the controller.

\subsection{First-Order One-Step-Ahead Adaptive Control}

Here we consider the 1st-order linear time-invariant plant
\begin{flalign}
y(t+1)
&=
a y(t) + b u(t) + w(t), \nonumber
\\
&=
{\underbrace{\begin{bmatrix}
  a \\ b
  \end{bmatrix}}_{=:{\theta^*}^\top } }^\top 
{\underbrace{\begin{bmatrix}
  y(t) \\ u(t)
  \end{bmatrix}}_{=:{\phi(t)} } }
   + w(t),
\quad y(t_0)=y_0.
\label{plant_first}
\end{flalign}
We have $y(t)\in\R $ as the output, $u(t)\in\R $ as the input, and 
$w(t)\in\R $ as the noise or disturbance. Here, $\theta^* $ is unknown but lies in {\bf closed and bounded set} ${\cal S}\subset\R^2 $; to ensure controllability we require that
 $\begin{bmatrix}
  a \\ 0
  \end{bmatrix}
  \notin {\cal S} $
for any $a\in\R $.
The control objective is to track a reference signal $y^* (t) $ asymptotically; we assume that we know it one step ahead.

In \cite{scl17} the case of ${\cal S} $ being convex is considered.
An adaptive controller is designed based on the ideal projection algorithm, and it is proven that a convolution bound is provided. In that paper this is leveraged to prove a degree of tolerance to time-variation and unmodelled dynamics, though the results there
are not quite as strong as those provided by Theorems 1--3.

Now we turn to the more general case of ${\cal S} $ not convex.
This was considered in \cite{cdc18} and a convolution bound was proven\footnote{Technically speaking, the bound \eqref{conv_v1}
was only proven for $t\geq \tau = t_0 $. However,
since the controller is time-invariant, the extension to 
$t\geq\tau\geq t_0 $ follows immediately.}, but nothing was proven about robustness to
time-variation and to unmodelled dynamics.
Here we will show that the controller proposed there fits into the framework of this paper, so that Theorems 1--3 can be applied.
In this case, it is proven in \cite{cdc18} that ${\cal S}$ can be covered by two convex and compact sets ${\cal S}_1 $ and ${\cal S}_2 $ so that, for every 
$
\begin{bmatrix}
a \\ b
\end{bmatrix}
\in {\cal S}_1 \cup {\cal S}_2
 $ 
 we have that $b\neq0 $.
 To proceed, we use two parameter estimators---one for ${\cal S}_1 $ and one for ${\cal S}_2 $---and then use a switching adaptive controller to switch between 
 the estimates as necessary.
For each $i\in \{1,2\} $ and given an estimate
 $\hat\theta_i(t)$ at time $t>t_0$, we have a prediction error of
\begin{equation*}
e_i(t+1):=
y(t+1)-{\hat\theta_i(t)}^\top\phi(t);
\end{equation*}
estimator updates are computed by
\begin{equation}
\check\theta_i(t+1)
=
\left\{
\begin{matrix*}[l]
\hat\theta_i(t) & & \text{if }\phi(t)=0
\\
\hat\theta_i(t)
+
\frac{\phi(t)}{\|\phi(t)\|^2}
e_i(t+1) & & \text{otherwise}
\end{matrix*}\right. 
\label{est2_a}
\end{equation} 
\begin{equation}
\hat\theta_i(t+1)=\text{Proj}_{{\cal S}_i}
\left\{
\check\theta_i(t+1)
\right\}.
\label{est2_b}
\end{equation}
We partition $\hat\theta_i(t)$ in a natural way by 
$\hat\theta_i(t)
=:
\begin{bmatrix}
  \hat a_i(t) \\ \hat b_i(t)
  \end{bmatrix}
 $.
We define a switching signal $\sigma: \Z\rightarrow \{1,2\} $ to choose which parameter estimates to use in the control law at any point in time. 
Namely, with $\sigma(t_0)\in\{1,2\} $, the choice is 
\begin{equation}
\sigma(t+1)
=
{\arg\min}_{i\in\{1,2\}} \; 
|e_i(t+1)|,
\qquad t\geq t_0,
\label{switch2}
\end{equation}
i.e. it is the index corresponding to the smallest prediction error.
Next we apply the Certainty Equivalence Principle to yield
 \begin{equation}
 u(t)
 =
 -\frac{\hat a_{\sigma(t)}(t)}{\hat b_{\sigma(t)}(t)}y(t)
 +
 \frac{1}{\hat b_{\sigma(t)}(t)} y^*(t+1).
 \label{control2}
 \end{equation}

We observe here that the controller \eqref{est2_a}--\eqref{control2} fits into the paradigm of Section 2: we set
\begin{flalign*}
{\cal X}
&=
{\cal S}_1 \times {\cal S}_2 \times \{1,2\},
\\
z_1(t)
&=\varnothing,
\\
z_2(t)
&=
\begin{bmatrix}
\hat \theta_1(t)
\\
\hat \theta_2(t)
\\ 
\sigma(t)
\end{bmatrix},
\\
r(t)
&=
y^*(t+1).
\end{flalign*}
In \cite{cdc18} it is proven that \eqref{est2_a}--\eqref{control2} 
provides a convolution bound for \eqref{plant_first};
by Theorems 1--3 we
immediately see that the same is true in the presence of time-variation and/or unmodelled dynamics.

\subsection{Pole-Placement Adaptive Control}

In this section, we consider the Pole-Placement Adaptive Control problem. 
We consider the $n^{\text{th}}$-order linear time-invariant plant
\begin{flalign}
y(t+1)
&=
\sum_{j=0}^{n-1} a_{j+1}y(t-j)
+\sum_{j=0}^{n-1} b_{j+1}u(t-j)
+w(t) 
\nonumber
\\
&=
\underbrace{\begin{bmatrix}
y(t) \\ \vdots \\ y(t-n+1) 
\\
u(t) \\ \vdots \\ u(t-n+1)
\end{bmatrix}^\top}_{=:\phi(t)^\top}
\underbrace{\begin{bmatrix}
a_1 \\ \vdots \\ a_n 
\\
b_1 \\ \vdots \\ b_n
\end{bmatrix}}_{=:\theta^*}
+w(t),
\qquad 
t\geq t_0
\label{pole_plant1}
\end{flalign}
with $\phi(t_0)=\phi_0$.
We have $y(t)\in\R $ as the output, $u(t)\in\R $ as the input, and 
$w(t)\in\R $ as the noise or disturbance. Here, $\theta^* $ is unknown but lies in a known set ${\cal S}\subset\R^{2n} $. Associated with this plant model are the polynomials
\[
  \mathbf A(z^{-1})=1-a_1 z^{-1}-a_2 z^{-2}\cdots-a_n z^{-n},\qquad \text{and }
\]
\[
 \mathbf B(z^{-1})=b_1 z^{-1}+b_2 z^{-2}\cdots+b_n z^{-n};
\]
We impose the following assumption:
\begin{assm}
${\cal S}$ is compact, and for each $\theta^*\in{\cal S} $,
the corresponding polynomials
$\mathbf A(z^{-1})$ and $\mathbf B(z^{-1})$
are coprime.
\end{assm}

The objective here is to obtain some form of stability with a secondary objective that of asymptotic tracking of a reference signal $y^*(t) $; 
the plant may be non-minimum phase, which limits the tracking goal. 

In \cite{pole18} the case of ${\cal S} $ convex is considered.
An adaptive controller is designed based on a modified version of
the ideal projection algorithm, and it is proven that a convolution bound is provided; this is leveraged there to prove a degree of tolerance to
time-variation and unmodelled dynamics, much like that provided by
Theorems 1 and 3.

Now we turn to the more general case of ${\cal S} $ not convex. 
This was also considered in \cite{pole18} subject to 
\begin{assm}
${\cal S}\subset {\cal S}_1 \cup {\cal S}_2 $ with ${\cal S}_1$ and ${\cal S}_2$ compact and convex, and for each $\theta^*\in{\cal S}_1 \cup {\cal S}_2 $,
the corresponding polynomials
$\mathbf A(z^{-1})$ and $\mathbf B(z^{-1})$
are coprime.
\end{assm}
\noindent A convolution bound was proven, but nothing was proven about robustness to time-variation and to unmodelled dynamics.
Here we will show that the controller proposed there fits into the
framework of this paper, so that Theorems 1--3 can be applied.
To proceed, we use two parameter estimators---one for ${\cal S}_1 $
and one for ${\cal S}_2 $, and then use a switching adaptive controller to switch between these estimates as necessary;
to prove that the approach works,
all closed-loop poles are placed at the origin.
 
The parameter estimation is projection-algorithm-based and similar to
that of the previous sub-section.
For $i\in\{1,2\} $ and given an estimate
 $\hat\theta_i(t)$ at time $t>t_0$, we have a prediction error of
\begin{equation*}
e_i(t+1):=
y(t+1)-\hat\theta_i(t)^\top\phi(t);
\end{equation*}
estimator updates are computed by
\begin{equation}
\label{est4_a}
\check\theta_i(t+1)
=
\left\{
\begin{matrix*}[l]
\hat\theta_i(t)
+
\frac{\phi(t)}{\|\phi(t)\|^2}
e_i(t+1) & & 
\text{if } \|\phi(t)\|\neq0
\\
\hat\theta_i(t) & & \text{otherwise}
\end{matrix*}\right. 
\end{equation} 
\begin{equation}
\hat\theta_i(t+1)=\text{Proj}_{{\cal S}_i}
\left\{
\check\theta_i(t+1)
\right\}.
\label{est4_b}
\end{equation}
We partition $\hat\theta_i(t)$ as 
$$\hat\theta_i(t)
=:
\begin{bmatrix}
  \hat a_{i,1}(t) & \cdots & \hat a_{i,n}(t) &
   \hat b_{i,1}(t) & \cdots & \hat b_{i,n}(t)
  \end{bmatrix}^\top;
 $$
associated with $\hat\theta_i(t)$ are the polynomials
\[
  \hat{\mathbf A}_i(t,z^{-1})=1-\hat a_{i,1}(t) z^{-1}-
  \hat a_{i,2}(t) z^{-2}\cdots-\hat a_{i,n}(t) z^{-n},\;\; \text{and }
\]
\[
 \hat{\mathbf B}_i(t,z^{-1})=\hat b_{i,1}(t) z^{-1}
 +\hat b_{i,2}(t) z^{-2}\cdots+\hat b_{i,n}(t) z^{-n}.
\]
We design a strictly proper controller 
by choosing its denominator and numerator polynomials, respectively, by
\[
  \hat{\mathbf L}_i(t,z^{-1})=1+\hat l_{i,1}(t) z^{-1}+
  \hat l_{i,2}(t) z^{-2}\cdots l_{i,n}(t) z^{-n},\quad \text{and }
\]
\[
 \hat{\mathbf P}_i(t,z^{-1})= p_{i,1}(t) z^{-1}
 + p_{i,2}(t) z^{-2}\cdots+ p_{i,n}(t) z^{-n}
\]
satisfying
\begin{equation}
\hat{\mathbf A}_i(t,z^{-1})\hat{\mathbf L}_i(t,z^{-1})
+
  \hat{\mathbf B}_i(t,z^{-1})\hat{\mathbf P}_i(t,z^{-1})
  =
  1,
  \label{sylv4}
\end{equation}
i.e. we place the closed-loop poles at zero.

A switching signal $\sigma: \Z\rightarrow \{1,2\} $ is used to choose which parameter estimates to use in the control law at any point in time. 
We update $\sigma(t)$ only every $N\geq2n$ steps; to this end, we define a sequence of switching times as follows: we initialize $\hat t_0:=t_0$ and then define 
$$\hat t_\ell:=t_0+\ell N,\; \ell\in\N.$$ 
The switching signal is given by
\begin{flalign}
\label{sw_const}
\sigma(t)=\sigma(\hat t_\ell),
\qquad t\in[\hat t_\ell,\hat t_{\ell+1}),\;\ell\in\Z^+.
\end{flalign}
Now define the control gains $\hat K_i(i)\in\R^{2 n}$ 
that are also only updated every $N\geq2n$ steps:
\begin{flalign}
\hat K_i(t)
&:=
[-\hat p_{i,1}(\hat t_\ell) \;\; \cdots \;\; -\hat p_{i,n}(\hat t_\ell) \;\;  
 -\hat l_{i,1}(\hat t_\ell) \;\; \cdots \;\; -\hat l_{i,n}(\hat t_\ell)]
 ,
\nonumber
\\
&\qquad
  t\in[\hat t_\ell,\hat t_{\ell+1}),\;\ell\in\Z^+;
 \label{control_para}
\end{flalign}
also define the filtered reference signal
\[
  r_2(t):=
  \sum_{j=1}^n
  \hat p_{\sigma(\hat t_\ell),j}(\hat t_\ell)
  y^*(t-j+1), 
  \;\; t\in[\hat t_\ell,\hat t_{\ell+1}),\;\ell\in\Z^+.
\]
For each $i$, define a performance signal
\begin{equation}
\label{perf1}
J_i(\hat t_\ell):=
\left\{
\begin{matrix*}
0 & \text{if } \phi(j)=0 
\text{ for all } j\in[\hat t_\ell,\hat t_{\ell+1})
\\
\max_{j\in[\hat t_\ell,\hat t_{\ell+1}),\phi(j)\neq0}
\frac{|e_i(j+1|}{\|\phi(j)\|}
& 
\text{otherwise}.
\end{matrix*}
\right.
\end{equation}
With $\sigma(\hat t_0)\in\{1,2\} $, we set 
\begin{equation}
\label{switch4}
\sigma(\hat t_{\ell+1} )=
{\arg\min}_{i\in\{1,2\}}
\,
J_i(\hat t_\ell),
\qquad
\ell\in\Z^+,
\end{equation}
and define the control law by
\begin{flalign}
u(t)
&=
\hat K_{\sigma(t-1)}(t-1)\phi(t-1)
+r_2(t-1).
\label{control4}
\end{flalign}

We observe here that the controller \eqref{est4_a},\eqref{est4_b},
\eqref{sylv4}, \eqref{perf1}, \eqref{switch4} and \eqref{control4} 
fits into the paradigm of Section 2; 
we can rewrite the controller 
in the form of \eqref{control}
as follows.
First we set
$${\cal X}
 =
 \R^{N} \times \R^N
 \times
 {\cal S}_1 \times {\cal S}_2
 \times  \{1,2\}.$$ 
For $t\geq t_0 $, we then set
\begin{flalign*}
z_1(t+1)
&=
\hat K_{\sigma(t)}(t)\phi(t)
+r(t),
\\
z_2(t)
&=
\begin{bmatrix}
z_{21}(t)
\\ 
z_{22}(t)
\\
\hat \theta_1(t)
\\
\hat \theta_2(t)
\\ 
\sigma(t)
\end{bmatrix},
\\
u(t)
&= z_1(t),
\end{flalign*}
with
$r(t)
=
r_2(t);$ 
for $t\geq t_0 $ and $i=1,2 $, we then set
\begin{equation}
z_{2i}(t+1) = 
{\scriptsize
\left[
\begin{matrix}
0 & 1 & & &
\\ 
 &  & 1 & &
\\
& & & \ddots &
\\
& & & & 1
\\
& & & & 0
\end{matrix}
\right]
} 
z_{2i}(t)
+
{\scriptsize
\left[
\begin{matrix}
0 \\ 0 \\ \vdots \\ 0 \\ 1
\end{matrix}
\right]
}
\times
\left\{
\begin{matrix*}[l]
\frac{|e_i(t+1) |}{\|\phi(t) \|} & \phi(t)\neq0
\\
0 & \text{otherwise},
\end{matrix*}
\right.
 \nonumber
\end{equation}
and for $t>t_0 $, we set\footnote{
Here we use $\|z_{2i}(t) \|_\infty $
to denote the $\infty $-norm of the vector $z_{2i}(t) $.
}
\begin{equation}
\sigma(t)=
\left\lbrace
\begin{matrix*}[l]
\sigma(t-1) & \frac{t-t_0}{N} \notin\N
\\ \\
{\arg\min}_{i\in\{1,2\}}
\underbrace{\|z_{2i}(t) \|_\infty }_{\qquad=J_i(t-N) }
& 
\frac{t-t_0}{N} \in\N.
\end{matrix*}
\right.
\nonumber
\end{equation}

\noindent
In \cite{pole18} it is proven\footnote{
Technically speaking, the bound \eqref{conv_v1}
is only proven for $t\geq\tau=t_0 $.
However, since the controller is periodic of period 
$N\geq 2n $,
it follows immediately that the same bound \eqref{conv_v1} holds
for $t\geq \tau\geq t_0$
for all $\tau\in\{t_0+N,t_0+2N, t_0+3N, \ldots \} $.
Since the controller has a bounded gain,
nothing untoward can happen for other $\tau$'s;
it is easy to prove that \eqref{conv_v1} will still hold
for a suitably larger choice of $c$ 
(but with the same $\lambda$).

}
that this adaptive controller
provides a convolution bound for \eqref{pole_plant1}; by Theorems 1--3
we see that the same is true in the presence of time-variation and/or unmodelled dynamics.

\section{Summary and Conclusion}

In this paper
we have shown that for
 a class of
nonlinear
plant and controller combinations,
if a convolution bound
on the closed-loop behavior
 can be proven, then
tolerance to small time-variations in the plant parameters
and a small amount of unmodelled dynamics follows immediately.
We applied the result to prove robustness of our recently designed multi-estimator switching adaptive controllers presented
in \cite{pole18} and \cite{cdc18}.
We expect this to be applicable to other adaptive control
paradigms, such as the adaptive control of nonlinear plants;
this will allow one to focus on the ideal plant in the analysis knowing
that robustness will come for free.
This result also has the potential to be applied in more general nonlinear contexts.

 \bibliographystyle{elsarticle-harv}
\bibliography{technote_scl1}

\begin{thebibliography}{9}
\expandafter\ifx\csname natexlab\endcsname\relax\def\natexlab#1{#1}\fi
\providecommand{\url}[1]{\texttt{#1}}
\providecommand{\href}[2]{#2}
\providecommand{\path}[1]{#1}
\providecommand{\DOIprefix}{doi:}
\providecommand{\ArXivprefix}{arXiv:}
\providecommand{\URLprefix}{URL: }
\providecommand{\Pubmedprefix}{pmid:}
\providecommand{\doi}[1]{\href{http://dx.doi.org/#1}{\path{#1}}}
\providecommand{\Pubmed}[1]{\href{pmid:#1}{\path{#1}}}
\providecommand{\bibinfo}[2]{#2}
\ifx\xfnm\relax \def\xfnm[#1]{\unskip,\space#1}\fi
\bibitem[{Desoer(1970)}]{desoer}
\bibinfo{author}{Desoer, C.}, \bibinfo{year}{1970}.
\newblock \bibinfo{title}{{Slowly varying discrete system $x_{i+1}=A_i x_i$}}.
\newblock \bibinfo{journal}{Electronics Letters} \bibinfo{volume}{6},
  \bibinfo{pages}{339--340}.
\bibitem[{Kreisselmeier(1986)}]{kreiss}
\bibinfo{author}{Kreisselmeier, G.}, \bibinfo{year}{1986}.
\newblock \bibinfo{title}{{Adaptive control of a class of slowly time-varying
  plants}}.
\newblock \bibinfo{journal}{Systems {\&} Control Letters} \bibinfo{volume}{8},
  \bibinfo{pages}{97--103}.
\bibitem[{Kreisselmeier and Anderson(1986)}]{kreiss86}
\bibinfo{author}{Kreisselmeier, G.}, \bibinfo{author}{Anderson, B.},
  \bibinfo{year}{1986}.
\newblock \bibinfo{title}{{Robust model reference adaptive control}}.
\newblock \bibinfo{journal}{IEEE Transactions on Automatic Control}
  \bibinfo{volume}{31}, \bibinfo{pages}{127--133}.
\bibitem[{Miller(2017a)}]{scl17}
\bibinfo{author}{Miller, D.E.}, \bibinfo{year}{2017}a.
\newblock \bibinfo{title}{{A parameter adaptive controller which provides
  exponential stability: The first order case}}.
\newblock \bibinfo{journal}{Systems {\&} Control Letters}
  \bibinfo{volume}{103}, \bibinfo{pages}{23--31}.
\bibitem[{Miller(2017b)}]{ccta17}
\bibinfo{author}{Miller, D.E.}, \bibinfo{year}{2017}b.
\newblock \bibinfo{title}{{Classical discrete-time adaptive control revisited:
  Exponential stabilization}}, in: \bibinfo{booktitle}{2017 IEEE Conference on
  Control Technology and Applications (CCTA)}, \bibinfo{publisher}{IEEE}. pp.
  \bibinfo{pages}{1975--1980}.
\bibitem[{Miller and Shahab(2018)}]{pole18}
\bibinfo{author}{Miller, D.E.}, \bibinfo{author}{Shahab, M.T.},
  \bibinfo{year}{2018}.
\newblock \bibinfo{title}{{Classical pole placement adaptive control revisited:
  linear-like convolution bounds and exponential stability}}.
\newblock \bibinfo{journal}{Mathematics of Control, Signals, and Systems}
  \bibinfo{volume}{30}, \bibinfo{pages}{19}.
\bibitem[{Miller and Shahab(2019)}]{acc19}
\bibinfo{author}{Miller, D.E.}, \bibinfo{author}{Shahab, M.T.},
  \bibinfo{year}{2019}.
\newblock \bibinfo{title}{{Classical d-Step-Ahead Adaptive Control Revisited:
  Linear-Like Convolution Bounds and Exponential Stability}}, in:
  \bibinfo{booktitle}{2019 American Control Conference},
  \bibinfo{publisher}{IEEE}, \bibinfo{address}{Philadelphia}.
\bibitem[{Shahab and Miller(2018)}]{cdc18}
\bibinfo{author}{Shahab, M.T.}, \bibinfo{author}{Miller, D.E.},
  \bibinfo{year}{2018}.
\newblock \bibinfo{title}{{Multi-Estimator Based Adaptive Control which
  Provides Exponential Stability: The First-Order Case}}, in:
  \bibinfo{booktitle}{2018 IEEE Conference on Decision and Control},
  \bibinfo{publisher}{IEEE}. pp. \bibinfo{pages}{2223--2228}.
\bibitem[{Zames(1966)}]{zames}
\bibinfo{author}{Zames, G.}, \bibinfo{year}{1966}.
\newblock \bibinfo{title}{{On the input-output stability of time-varying
  nonlinear feedback systems Part one: Conditions derived using concepts of
  loop gain, conicity, and positivity}}.
\newblock \bibinfo{journal}{IEEE Transactions on Automatic Control}
  \bibinfo{volume}{11}, \bibinfo{pages}{228--238}.

\end{thebibliography}

\end{document}